\documentclass{article}
\usepackage[latin9]{inputenc}
\usepackage{a4}
\usepackage{amsfonts}
\usepackage{amssymb}
\usepackage{amsmath}
\usepackage{amsthm}
\usepackage{nomencl}
\begin{document}
\pagenumbering{arabic}
\renewcommand{\nomname}{Symbolverzeichnis}
\setlength{\nomlabelwidth}{.25\hsize}
\renewcommand{\nomlabel}[1]{#1 \dotfill}
\setlength{\nomitemsep}{-\parsep}
\newtheorem{satz}{Theorem}[section]
\newtheorem{lemma}[satz]{Lemma}
\newtheorem{prop}[satz]{Proposition}
\newtheorem{kor}[satz]{Corollary}
\theoremstyle{definition}
\newtheorem{defi}{Definition}[section]
\title{Solvability of Generalized Monomial Groups}
\date{}
\author{Joachim K\"onig\\
{\normalsize Institut f\"ur Mathematik der Universit\"at W\"urzburg}\\
{\normalsize Am Hubland, 97074 W\"urzburg} \\ 
{\normalsize joachim.koenig@mathematik.uni-wuerzburg.de}}
\maketitle

\abstract{The solvability of monomial groups is a well-known result in character theory. Certain properties of Artin $L$-series suggest a generalization of these groups, namely to such groups where every irreducible character has some multiple which is induced from a character $\varphi$ of $U$ with solvable factor group $U/ker(\varphi)$. Using the classification of finite simple groups, we prove that these groups are also solvable. This means in particular that the mentioned properties do not enable one to deduce a proof of the famous Artin conjecture for any non-solvable group from a possible proof for solvable groups.}
\renewcommand{\contentsname}{Table of contents}
\tableofcontents
\newpage
\section{Introduction}
Let $G$ be a finite group\footnote{Throughout this paper all groups will be assumed to be finite.}. Then a character $\chi$ of $G$ is called monomial if $\chi = \phi^G$ for a linear character of some subgroup $U \leq G$. If all irreducible characters of $G$ are monomial, $G$ is called monomial.\\
A classical theorem by Taketa (see for example \cite{Is}, Cor. (5.13)) states that monomial groups are solvable.\\
There are several generalizations of Taketa's result, e.g. the statement that if all the inducing subgroups $U$ have the property that the composition factors of $U/ker(\varphi)$ are in a given set of simple groups, then also the composition factors of $G$ lie in this set. In particular, Taketa's argument can easily be generalized to the case that $U/ker(\varphi)$ is always solvable.\\
\\
In this paper, we will consider a different generalization, namely the case that only some multiple of the character $\chi$ is "solvably induced":
\begin{defi}[\bf{QSI-Group}]
\label{qsi}
\ \\
Let $G$ be a finite group, $\chi \in Char(G)$. Then $\chi$ is called QSI (quasi solvably induced) from $U \leq G$ and $\varphi \in Irr(U)$, if there exists $k \in \mathbb{N}$ such that $k \chi = \varphi^G$, and $U/ker(\varphi)$ is solvable. \\
$G$ is called QSI, if all $\chi \in Irr(G)$ are QSI.
\vspace{4mm}
\end{defi}
\ \\
By definition, monomial groups are a special case of QSI groups, because there we always have $k=1$ and $\varphi$ linear, i.e. $U/ker(\varphi)$ abelian.\\
It is trivial that even all solvable groups are QSI (simply set $U := G$, $\varphi := \chi$ and $k := 1$). This paper aims at proving the converse of that statement, i.e.:
\\
\begin{satz}
\label{Hauptsatz}
Let $G$ be a finite QSI-group. Then $G$ is solvable.
\end{satz}
\vspace{3mm}
\ \\
The phenomenon of "quasi-induced" characters in the sense of definition \ref{qsi} arises in a natural way in the study of Artin $L$-functions:\\
\\
Let $F|K$ be a Galois extension of number fields with group $G$, and $\chi \in Char(G)$ a character. Then one can define a certain number-theoretic function $L(s, \chi, F|K)$, called an Artin $L$-function, (where  $s \in \mathbb{C}$, $Re(s) > 1$). Artin $L$-functions are a generalization of Dirichlet $L$-functions, and as with the latter, the question arises if they can be continued homomorphically to $\mathbb{C}$.\\ \\
Artin's conjecture states that this is true for all finite groups $G$ and all $\chi \in Char(G)$ with $(\chi, 1) = 0$. (If $\chi = 1$, then the Artin $L$-function is just the zeta function of $K$, and therefore meromorphic with a simple pole at $s=1$.)\\
From the functional equation of Artin $L$-functions it follows that $L(s, \sum_{i_1}^n{\chi_i} , F|K) = \prod_{i=1}^n{L(s, \chi_i, F|K)}$, so it suffices to prove Artin's conjecture for all $1 \neq \chi \in Irr(G)$.\\ \\
Still the conjecture is only known to be true for certain special cases, e.g. for $1 \neq \chi$ linear.\\
It is also known that for a subgroup $U \leq G$ and a character $\varphi \in Char(U)$ there is a relation between the $L$-functions of $\varphi$ and of $\varphi^G$, namely 
\[L(s, \varphi^G, F|K) = L(s, \varphi, F|F^U)\]
where $F^U$ is the fixed field of $U$.\\
\\
Brauer proved that for the general case there is at least a meromorphic continuation on $\mathbb{C}$.\\
The proof is based mainly on Brauer's induction theorem:
\begin{itemize}
\item[] Let $1 \neq \chi \in Irr(G)$, then $\chi = \sum_{i=1}^{r}{z_i \cdot \varphi_i^G}$, where the $\varphi_i$ are linear characters of subgroups $U_i$ of $G$, and $z_i \in \mathbb{Z}$. But it is well-known that 
\[L(s, \chi, F|K) = \prod_{i}{(L(s, \varphi_i^G, F|K))^{z_i}}\]
and by the above properties each factor of this product is obviously meromorphic.
\end{itemize}
\vspace{5mm}
These first results can be used to prove the following
\begin{lemma}
Let $F|K$ be a Galois extension of number fields with group $G$, and $\chi \in Char(G)$ a {\bf quasi-monomial} character, i.e. $k \cdot \chi = \varphi^G$ for some $U \leq G$, $\varphi \in Irr(U)$ linear, and $k \in \mathbb{N}$.\\
If $(\chi, 1) = 0$, then $L(s, \chi, F|K)$ has a holomorphic continuation on $\mathbb{C}$.
\end{lemma}
\begin{proof}
Obviously $\varphi \neq 1$. Therefore $L(s, \chi, F|K)^k = L(s, k \cdot \chi, F|K) = L(s, \varphi^G, F|K)  = L(s, \varphi, F|F^U)$ has a holomorphic continuation on $\mathbb{C}$, and together with Brauer's result mentioned above the same follows for  $L(s, \chi, F|K)$.
\end{proof}
\vspace{3mm}
\ \\
This argument obviously works whenever Artin's conjecture is known to be true for $L(s, \varphi, F|F^U)$.\\
For example, Langlands and Tunnell (\cite{La}, \cite{Tu}) proved the conjecture for characters $\varphi$ degree 2, with $U/ker(\varphi)$ solvable. In this case $U/ker(\varphi)$ is a cyclic central extension of a finite solvable subgroup of $PGL(2, \mathbb{C})$, and it is well-known that these finite subgroups are either cyclic, dihedral or isomorphic to $A_4$ or $S_4$.\\ \\
\ \\
The result of theorem \ref{Hauptsatz} can therefore be interpreted in the following way:\\
Even if Artin's conjecture were known to be true for all solvable groups, the methods shown above will not enable one to prove the conjecture for a single non-solvable group.\\
\vspace{4mm}

\section{Reduction to finite simple groups}
As the proof of Taketa's theorem essentially uses the fact that $k=1$ always holds, a proof of theorem \ref{Hauptsatz} requires different ideas.\\
In this chapter we will reduce the general problem outlined in section 1 to the investigation of finite simple groups. After this we can attack the problem case by case, thanks to the classification of finite simple groups.\\
\\
First, a simple observation:
\begin{lemma}
If $G$ is QSI, and $N \trianglelefteq G$, then also $G/N$ is QSI.
\end{lemma}
\begin{proof}
Every irreducible character $\widehat{\chi}$ of $G/N$ yields an irreducible character $\chi$ of $G$, with $N \trianglelefteq  ker(\chi)$. Now $\chi$ is QSI from $U \leq G$, i.e. $k \cdot \chi = \varphi^G$ for some $\varphi \in Irr(U)$. Thus $N \trianglelefteq ker(\varphi^G) = \bigcap_{g \in G}{ker(\varphi)^g}$,\\
i.e. $N \trianglelefteq ker(\varphi)$, and $\varphi$ corresponds to a character $\widehat{\varphi} \in Irr(U/N)$. But \[\widehat{\varphi}^{G/N}(xN) = \frac{1}{|U/N|} \cdot \sum_{\stackrel{yN \in G/N}{(xN)^y \in U/N}}{\varphi((xN)^y)} =\] \[ = \frac{1}{|U|} \cdot \sum_{\stackrel{y \in G}{x^y \in U}}{\varphi(x^y)} = \varphi^G(x) = k \cdot \chi(x) = k \cdot \widehat{\chi}(xN)
\]
and therefore $\widehat{\chi}$ is QSI.
\end{proof}
\vspace{10mm}
\begin{lemma}
Let $G$ be a minimal counterexample to the assertion of Theorem \ref{Hauptsatz}. Then $N \trianglelefteq G \leq Aut(N)$, where $N = S^m$, and $S$ is a non-abelian simple group.
\end{lemma}
\begin{proof}
Let $G$ be such a minimal counterexample, i.e. $G$ is finite, nonsolvable, but QSI. Let $N$ be a minimal normal subgroup of $G$. Then $N = S^m$ for some finite simple group $S$. By the minimality of $G$, $G/N$ must be solvable (as $G/N$ is also  QSI). Therefore $N$ is nonsolvable, which means that $S$ is non-abelian. If there existed another minimal normal subgroup $\widehat{N}$, then $N \cap \widehat{N} = 1$, so $\widehat{N} \hookrightarrow G/N$, contrary to the solvability of $G/N$. So $N$ must be the unique minimal normal subgroup of $G$. As $C_G(N) \trianglelefteq G$, we have either $C_G(N) = 1$ or $C_G(N) \geq N$. The latter is obviously impossible.\\
So $G = N_G(N)/C_G(N) \leq Aut(N)$.
\end{proof}
\vspace{4mm}
\ \\
In the following, we will look at such possible minimal counterexamples $G$. Note that if $N = S^m$ with a non-abelian finite simple group $S$, then $Aut(N)$ is a semidirect product:\\ $Aut(N) \cong Aut(S)^m:Sym(m)$.\\ \\
With the following simple argument, which basically uses a special case of Mackey's decomposition formula, we can derive from the QSI-property of $G$ at least the QSI-property of certain characters of a normal subgroup $N$ of $G$:\\
\begin{lemma}
Let $N \triangleleft G$, and $\chi \in Irr(N)$ be a $G$-invariant character (e.g. a character which is uniquely determined by its degree), such that there exists a QSI character $\rho \in Irr(G)$ with $(\rho_{|N}, \chi) \neq 0$.\\
Then $\chi$ is QSI, too.
\end{lemma}
\begin{proof}
Let $\rho \in Irr(G)$ with $(\rho_{|N}, \chi) \neq 0$, then by Clifford's theorem $\rho_{|N} = k \cdot \chi$ for some $k \in \mathbb{N}$.\\
Let $U \leq G$ and $\varphi \in Irr(U)$ such that $\rho$ is QSI from $U$ and $\varphi$, so $(\varphi_U^G)_{|N} = k \cdot m \cdot \chi$, for some $m \in \mathbb{N}$.\\
But that means for $x \in N$ and $g_1, ... g_n$ representatives of cosets of $N$:
\[km \cdot \chi(x) = \frac{1}{|U|} \sum_{i=1}^n{\sum_{\stackrel{y \in N}{x^{g_iy} \in \ U \cap N}}{\varphi(x^{g_iy})}} = \frac{|U \cap N|}{|U|} \sum_{i=1}^{n}{(\varphi_{|U \cap N}^N)^{(g_i^{-1})}(x)},\]
and as the sum is a multiple of $\chi$, the same must hold for any of its components, e.g. $(\varphi_{|U \cap N})^N$ is a multiple of $\chi$.\\
But that means that $\chi$ is QSI from $U \cap N$.
\end{proof}
\vspace{3mm}
\ \\
In our situation, with $S^m \leq G \leq Aut(S)^m:Sym(m)$, we can argue in two steps: Firstly if $\chi$ is an $Aut(S)$-invariant irreducible character of $S$, then $\chi^m$ is $G$-invariant in $N = S^m$. So if $G$ is QSI, then so is the character $\chi^m$, and therefore also $\chi$.\\
So now all we need to do is to find for each non-abelian simple group $S$ an $Aut(S)$-invariant irreducible character that cannot be QSI.\\
\\
In order to prove that a subgroup $U$, as required, does not exist, we need one more simple observation:
\begin{lemma}
Let $\chi \in Irr(G)$ be QSI from $U$ and $\varphi$; then:
\begin{enumerate}
\item $U$ intersects every class on which $\chi$ takes a non-zero value. More precisely, the fraction of the class of $g \in G$ which is contained in $U$ must be at least $\frac{|\chi(g)|}{\chi(1)}$.
\item If $U$ is non-abelian simple, then $\chi = 1_G$.
\end{enumerate}
\end{lemma}
\begin{proof}
1) is clear by the definition of an induced character.\\
2) If $U$ is non-abelian simple, then $\varphi = 1_{U}$, hence $1_{G}$ is a component of $\varphi^G$, and so $\chi = 1_G$ (and more precisely $U=G$).
\end{proof}
\section{Proof for the case $A_n$} 
\begin{lemma}
Let $S = A_n$, $n \geq 5$, $n \neq 6$. $S^m \leq G \leq Aut(S^m)$. Then $G$ is not QSI.
\end{lemma}
\begin{proof}

Assume that $G$ is QSI.\\
To obtain a contradiction we only need to find an $Aut(S)$-invariant irreducible character of $S$ which is not QSI.
As we have excluded the case $n=6$ (we will treat $A_6 \cong PSL_2(9)$ later as a group of Lie type), it always holds that $Aut(S)= S_n$, and therefore the character $\chi_n := \pi_n -1$ of degree $n-1$, with $\pi_n$ the permutation character of the natural action, is $Aut(S)$-invariant.\\
\\
For the case $n=5$, it's easy to see that a subgroup with the required properties does not exist (e.g. we would need elements of both orders 5 and 3).\\
In all other cases the 2-point stabilizer $A_{n-2}$ is still doubly transitive. Now assume the character $\chi_{n} \in Irr(A_{n})$ were QSI, i.e. $k \cdot \chi_{n} = \varphi^{A_{n}}$ for some $\varphi \in Irr(U)$, $U < A_{n}$, and $U/ker(\varphi)$ solvable.\\
But $(\chi_{n})_{|A_{n-2}} = \chi_{n-2} + 2\cdot 1_{A_{n-2}}$, and so by Mackey's formula
\[k \cdot \chi_{n-2} + 2k \cdot 1_{A_{n-2}} = (\varphi^{A_{n}})_{|A_{n-2}} = \sum_{t \in T}{(\varphi^{(t)}_{|A_{n-2} \cap U^t})^{A_{n-2}}}\]
where the $t$ are representatives of double cosets of $A_{n-2}$ and $U$.\\
Therefore also $(\varphi_{|A_{n-2} \cap U})^{A_{n-2}} = k_1 \cdot \chi_{n-2} + k_2 \cdot 1_{A_{n-2}}$, with $k_1, k_2 \in \mathbb{N}_0$, and by replacing $U$ with some $U^t$ we can assume w.l.o.g that $k_2 > k_1$, i.e. this character sum takes only positive values. But that means that the inducing subgroup $A_{n-2} \cap U$ has to intersect all conjugacy classes of $A_{n-2}$, which is only possible for $U \geq A_{n-2}$. We already know that $U$ cannot be non-abelian simple, therefore $U \cong S_{n-2}$. But this is impossible too, e.g. the conjugacy classes of $(1...n)$ or $(1...n-3)(n-2, n-1, n)$ (depending on which of those permutations is even) would have to intersect $U$.
\end{proof}
\vspace{10mm}
\section{Proof for the sporadic cases}
In all cases the following procedure already succeeds:\\ 
Find an $Aut(S)$-invariant irreducible character $\chi \neq 1$ with non-zero values on elements of certain orders. If possible these element orders are chosen so that there is no maximal subgroup of $S$ with order divisible by all of them, which already contradicts QSI-property.\\
If maximal subgroups with the required order exist but are always non-abelian simple we can descend further and hope to get a contradiction in the next step.\\
The character degrees, element orders, and, if they exist, the maximal subgroups $M$ containing elements of all these orders, are given in table \ref{T1}.\\
\begin{table}
\caption{Sporadic simple groups}
\label{T1}
\renewcommand{\arraystretch}{1.2}
\begin{minipage}{7cm}
		\begin{tabular}[h]{|cccc|}
		\hline
			{S} & {$\chi(1)$} & {element orders} & {$M$}\\
			\hline
			$M_{11}$ & 45 & 8, 11 & -\\
			$M_{12}$ & 54 & 10, 11 & -\\
			$M_{22}$ & 21 & 8, 11 & -\\
			$M_{23}$ & 22 & 7, 23 & -\\
			$M_{24}$ & 3520 & 21, 23 & -\\
			$J_1$ & 77 & 5, 19 & -\\
			$J_2$ & 36 & 5, 7 & -\\
			$J_3$ & 324 & 17, 19 & -\\
			$J_4$ & 889111 & 35, 37 & -\\
			$HS$ & 3200 & 7, 11 & $M_{22}$\\
			$McL$ & 4500 & 7, 11 & $M_{22}$\\
			$He$ & 1920 & 7, 17 & -\\
			$Ru$ & 102400 & 13, 29 & -\\
			$Suz$ & 248832 & 11, 13 & -\\
\hline
		\end{tabular}
\end{minipage}
\begin{minipage}{7cm}
		\begin{tabular}[h]{|cccc|}
		\hline
			{S} & {$\chi(1)$} & {element orders} & {$M$}\\
			\hline
			$ON$ & 26752 & 7, 31 & -\\
			$Co_{3}$ & 275 & 7, 23 & $M_{23}$\\
			$Co_{2}$ & 275 & 7, 23 & $M_{23}$\\
			$Co_{1}$ & 21049875 & 13, 23 & -\\
			$Fi_{22}$ & 1360800 & 11, 13 & -\\
			$Fi_{23}$ & 30888 & 17, 23 & -\\
			$Fi_{24}'$ & 57477 & 29, 33 & -\\
			$HN$ & 2985984 & 11, 19 & -\\
			$Ly$ & 1534500 & 37, 67 & -\\
			$Th$ & 30875 & 27, 31 & -\\
			$B$ & 96255 & 47, 48 & -\\
			$M$ & 18538750076 & 59, 119 & -\\
			$Ti$ & 1728 & 5, 13 & $PSL(2,25)$\\
			\ & \ & \ & \ \\
\hline
		\end{tabular}
	  \\
\end{minipage}
		\end{table}
In each case the characters, and if possible the maximal overgroups of the chosen elements, were taken from \cite{ATLAS} and checked with the computer programs GAP and/or MAGMA. \\
In addition, in several cases (but not in all, as our choice of elements shows), table 4 from \cite{GK} was used to obtain the maximal subgroups containing one of the specified elements. Often those subgroups are normalizers or other small subgroups, so that the number of conjugacy classes not contained there is so large that the second class can almost be chosen arbitrarily.\\	
Note also that we include the Tits group $Ti := {^2}F_4(2)'$ here, and not in the Lie type section, as $Ti$ has no unique (Steinberg) character of degree $|Ti|_2$ (and, strictly speaking, is not a Lie type group). 
\vspace{10mm}
\section{Proof for the classical groups}
In this section let $S = S(n,p^r)$ be a non-abelian simple classical group over a field of characteristic $p$. For these groups, and more generally for all simple groups of Lie type, there is a unique irreducible character, called the Steinberg character $St$, of degree $|S|_p$, where the latter is the highest power of $p$ dividing $|S|$. Moreover, $St(g) \neq 0$ whenever $g$ is a $p'$-element \footnote{For the properties of the Steinberg character, cf. \cite{Fe}}.
\\
\\
We will prove that (with few exceptions) $St$ cannot be QSI by looking at the maximal overgroups of certain $p'$-elements, and concluding that elements of some other order coprime to $p$ cannot be contained there.\\
\\
To this aim, recall the following results about some special elements of finite classical groups:
\begin{lemma}
In $GL_n(q)$ there exists a cyclic subgroup $T$ of order $q^n-1$ acting transitively on the non-zero vectors of $V$ - and in particular irreducibly on $V = (\mathbb{F}_q)^n$ . Every irreducible cyclic subgroup of $GL_n(q)$ is contained in such a group $T$.\\
$T$ is called a Singer cycle.\\
\\
In general, if $G \leq GL_n(q)$ is a classical group over $V$, and $T \cap G$ acts irrreducibly on $V$, then $T \cap G$ is called a Singer cycle of $G$.
\\
Singer cycles exist in $SL_n(q)$ (and have order $\frac{q^n-1}{q-1}$), $Sp_{2n}(q)$ (order $q^n+1$), $O_{2n}^-(q)$ (order $q^n+1$) and for $n$ odd also in $GU_n(q)$ (order $q^n+1$ as well).In contrast, the groups $GU_n(q)$ for $n$ even, $O_{2n}^+(q)$ and $O_{2n+1}(q)$ do not contain Singer cycles\footnote{cf. \cite{H3}}.
\end{lemma}
\ \\
Moreover the following well-known result (see \cite{Z1}) will be of vital importance for our subsequent considerations:
\begin{satz}[\bf{Zsigmondy}]
Let $d, n \in \mathbb{N}$, and both $\geq 2$. Then $d^n - 1$ has a {\itshape primitive prime divisor} \ (i.e. a prime divisor dividing none of the terms $d^k - 1$, $\mathbb{N} \ni k < n$), except for the following cases:
\begin{itemize}
	\item $d = 2, n = 6$
	\item $d$ a Mersenne prime , $n = 2$.
\end{itemize}
\end{satz}

\vspace{4mm}
 \ \\
As an immediate consequence we get
\begin{lemma}
\label{zslem}
Let $d, n \in \mathbb{N}$, both $\geq 2$, and $p_n$ a primitive prime divisor of $d^n-1$.\\
Then $p_n \equiv 1 (\text{mod n})$.\\
If at the same time $p_n$ divides $d^m-1$, then: \ $n | m$.
\begin{proof}
The first part follows from the fact that $d$ obviously has multiplicative order $n$ modulo $p_n$, and thus $\varphi(p_n) = p_n-1$ is divisible by $n$.\\
With $d^n-1$ and $d^m-1$, $p_n$ also divides $d^{m-n}-1$. By induction, this yields the second part of the assertion.
\end{proof}
\end{lemma}
\vspace{3mm}
 \ \\
Note that an element of $GL_n(q)$ with order divisible by a primitive prime divisor of $q^n-1$ is necessarily irreducible (by Maschke's theorem).\\
\\
Now let $S$ be a finite simple classical group, and $\widehat{S}$ the natural pre-image of $S$ in $GL(V)$, i.e. $\widehat{S} = SL(V)$, $Sp(V)$, $SU(V)$ or $\Omega^{(+/-)}(V)$ respectively. Furthermore let $\widehat{x}$ be an element of a maximal torus $\widehat{T}$ of $\widehat{S}$, as given in table \ref{T3}.\\
Note that $S = S_{sc}/Z(S_{sc})$, where $S_{sc}$ is the simply-connected version of the respective group of Lie type.\\
The orders of all maximal tori of $S_{sc}$ are well-known, see e.g. Carter(\cite{Ca2}, Propositions 7-10). For our specific selection cf. \cite{MSW}, table 1.1.\\
\\
\begin{table}[h]
\caption{Choice of elements for the classical groups}
\label{T3}
\renewcommand{\arraystretch}{1.3}
		\begin{tabular}[h]{|l|c|c|}
		\hline
			{S} & {$ord(\widehat{x})$} & {$|\widehat{T}|$}\\
			\hline
			$PSL_n(q)$ & $\frac{q^n-1}{q-1}$ & $ord(\widehat{x})$\\
			$PSp_{2n}(q)$ & $q^n+1$ & $ord(\widehat{x})$\\
			$PSU_n(q)$, $n$ odd & $\frac{q^n+1}{q+1}$ & $ord(\widehat{x})$\\
			$PSU_n(q)$, $n$ even & $\frac{q^{n-1}+1}{q+1}$ & $ord(\widehat{x}) \cdot (q+1)$\\
			$P\Omega_{2n+1}(q)$ & $\frac{1}{2} \cdot (q^n+1)$ & $ord(\widehat{x})$\\
			$P\Omega_{2n}^-(q)$ & $\frac{1}{(2,q-1)} \cdot (q^n+1)$ & $ord(\widehat{x})$\\
			$P\Omega_{2n}^+(q)$ & $\frac{1}{(2,q-1)} \cdot (q^{n-1}+1)$ & $ord(\widehat{x}) \cdot (q+1)$\\
			\hline
		\end{tabular}
		\end{table}
		\\
\vspace{6mm}	
\\	
In the cases $PSL_n(q)$, $PSU_n(q)$ with $n$ odd, $PSp_{2n}(q)$ and $P\Omega_{2n}^-(q)$, $\widehat{x}$ is a  Singer element of $\widehat{S}$.\\
In the remaining cases we use Singer elements of certain subgroups.\\
For the case $S = PSU_n(q)$ we make use of the embedding $SU_{n-1}(q) \bot SU_{1}(q) \leq SU_n(q)$.\\
For $S = P\Omega_{2n+1}(q)$, or $P\Omega_{2n+2}^+(q)$ respectively, we choose a Singer element of $\Omega_{2n}^-(q)$, using for this purpose  the embeddings  $\Omega_{2n-2}^-(q) \bot \Omega_{2}^-(q) \leq \Omega_{2n}^+(q)$, and for $q$ odd:\\ $\Omega_{2n}^-(q) \bot \Omega_1(q) \leq \Omega_{2n+1}(q)$.
\\ \\
We now consider the maximal overgroups of the elements $x$.\\
According to Aschbacher's classification (\cite{As1}), a maximal subgroup of a simple classical group belongs either to one of eight classes $\mathcal{C}_1$, ..., $\mathcal{C}_8$ (whose representatives can in each case be viewed as stabilizers of certain structures), or to a set $\mathcal{S}$ of exceptional cases\footnote{for the detailed notation of these Aschbacher classes cf. \cite{As1} or \cite{KL}, chapter 4}.\\
The respective maximal overgroups of $x$ are given in \cite{MSW}, Theorem 1.1 (and partly in \cite{Be1}).
In groups with Singer elements those are mostly subgroups of Aschbacher class $\mathcal{C}_3$  ("field extension type"), in the other groups mainly subgroups of class $\mathcal{C}_1$ (reducible subgroups).
\\
\\
The following theorem summarizes the results of \cite{Be1} (for $PSL_n(q)$) and \cite{MSW}, Theorem 1.1 (for the remaining classical groups):\\
\begin{satz}
\label{msw}
Let $S$ be a non-abelian simple classical group, and $x \in S$ as chosen in table \ref{T3}. Let $M < S$ be a maximal subgroup with $x \in M$. Then one of the following holds:
\begin{enumerate}
\item $M \in \mathcal{C}_1 \cup ... \cup \mathcal{C}_8$, as defined in \cite{As1}. More precisely one gets:\\
\renewcommand{\arraystretch}{1.3}
		\begin{tabular}{|l|c|c|}
		\hline
			{S} & {$M \in ...$} & {possible types of $M$}\\
			\hline
			$PSL_n(q)$ & $C_3$ & $GL_{n/r}(q^r)$ \ ,with $1 \neq r|n$\\
			$PSp_{2n}(q)$ & $C_3 \cup C_8$ & $Sp_{2n/r}(q^r)$, $GU_n(q)$ or $O_{2n}^{-}(q)$\\
			$PSU_n(q)$, $n$ odd & $C_3$ & $GU_{n/r}(q^r)$\\
			$PSU_n(q)$, $n$ even & $C_1$ & $GU_1(q) \bot GU_{n-1}(q)$\\
			$P\Omega_{2n+1}(q)$ & $C_1$ & $O_1(q) \bot O_{2n}^-(q)$\\
			$P\Omega_{2n}^-(q)$ & $C_3$ & $O_{2n/r}^-(q^r)$ or $GU_{n}(q)$\\
			$P\Omega_{2n}^+(q)$ & $C_1 \cup C_3$ & $GU_{n}(q)$, $O_n(q^2)$, $O_1(q) \bot O_{2n-1}(q)$\\
			\ & \ & or $O_2^-(q) \bot O_{2n-2}^-(q)$\\
			\hline
		\end{tabular}
		\\
		\\
		\item $S = PSL_n(q)$, with $(n,q) \in \left\{(2,5), (2,7), (2,9), (3,4)\right\}$.
		\item $S = PSU_4(2)$.
		\item $(M,S) \cong (PSL_2(7), PSU_3(3)), (A_7, PSU_3(5))$, $(A_7, PSU_4(3))$,\\  $(PSL_3(4), PSU_4(3))$, $(PSL_2(11), PSU_5(2))$, $(M_{22}, PSU_6(2))$, \\ $(S_9, P\Omega_7(3)), (PSL_2(17), PSp_8(2))$, $(2^4.A_5, PSp_4(3)),$ or $(A_9, P\Omega_8^+(q))$.
		\item $S = P\Omega_8^+(q)$, and either $soc(M) \cong P\Omega_7(q)$ with $q$ odd, or $soc(M) \cong PSp_6(q)$ with $q$ even.
\end{enumerate}
\end{satz}
\vspace{6mm}
Now it is relatively easy to prove that the Steinberg character of a non-abelian simple group of Lie type cannot be QSI (with a small number of exceptions).\\
First we consider the exceptional cases:
\begin{enumerate}
\item $S = PSL_n(q)$, with $(n,q) \in \left\{(2,5), (2,7), (2,9), (3,4)\right\}$\\
$PSL_2(5) \cong A_5$ has already been treated.\\
In $PSL_2(9) \cong A_6$, as the Steinberg character takes non-zero values on all elements of order prime to 3, we need a subgroup with order divisible by 20. But that only leaves $A_5$, which is non-abelian simple.\\
In $PSL_3(4)$ there are no subgroups containing elements of orders 3,5 and 7 simultaneously.\\
Finally, $PSL_2(7) \cong PSL_3(2)$ has two Steinberg characters. However, both are monomial, so here for once we have to deviate with our choice for $\chi$. There is a unique character of degree 6, with non-zero values on elements of order 2 and 7, and a maximal subgroup with order divisible by 14 does not exist.
\item In $S = PSU_4(2)$, no proper subgroup contains elements of orders 9 and 5 simultaneously.
\item  $S = PSU_3(3), M = PSL_2(7)$. Then $M$ is non-abelian simple, so we have to look only at proper subgroups, none of which contain both the divisors 2 and 7.\\
	 $S = PSU_3(5), M = A_7$. Here $S$ contains elements of order 8, $A_7$ obviously does not.\\
	 $S = PSU_5(2), M = PSL_2(11)$. Once again we need to descend further, and no proper subgroup of $PSL_2(11)$ contains the divisors 3, 5 and 11 simultaneously. \\
	 $S = PSU_4(3)$, $M \in \left\{A_7, PSL_3(4)\right\}$: In both cases elements of order 8 are missing.\\
	 $S = PSU_6(2)$, $M = M_{22}$: Elements of order 15 are not contained.\\
	$S = P\Omega_7(3)$, $M = S_9$: $|S|$ is divisible by 13. \\
	$S = PSp_8(2)$, $M = PSL_2(17)$: Here, the prime divisors 5 and 7 of $|S|$ are missing in $|M|$.\\
	 $S = PSp_4(3)$, $M = 2^4.A_5$: Note that $PSp_4(3) \cong PSU_4(2)$ \footnote{This change of characteristic is important, as the 3-Steinberg character is indeed QSI from a subgroup of $M$ of order 160}.\\
 	$S = P\Omega_{8}^+(2)$, $M = A_9$: No proper subgroup of $A_9$ contains elements of orders 7 and 9 simultaneously.
\item  $S = P\Omega_8^+(q)$, and either $soc(M) \cong P\Omega_7(q)$ with $q$ odd, or $soc(M) \cong PSp_6(q)$ with $q$ even:\\
This is a special case only with regard to Aschbacher classification, but not for our purposes, as we only need order arguments. Hence these cases can be treated like the standard ones, which we will consider from now on.
\end{enumerate}
It often suffices to use order arguments in the following way: many of the maximal overgroups $M$ of $x$ lose primitive prime divisors of $q^d-1$ for some $d \in \mathbb{N}$, so they are not in question for inducing the Steinberg character. In the few cases where the respective primitive prime divisor does not exist, the structure of $M$ may show that certain elements cannot be contained. A similar argumentation is already contained in \cite{MSW}, pp. 110 ff.\footnote{However in some cases we have to argue in a different way as \cite{MSW}, as we must not use any $p$-singular elements to generate $S$}\\  \\
Additionally we can use another remark about the Steinberg character $St$ of a Lie type group:
\begin{lemma}
Let $S = S(n,p^r)$ be a simple group of Lie type with defining characteristic $p$, and assume that its Steinberg character $St$ is QSI from $U$ and $\varphi$. Then $p \not| \  |ker(\varphi)|$.
\end{lemma}
\begin{proof}
$|S|_p = St(1) \ | \ [S:U] \cdot \varphi(1) \ | \ [S:ker(\varphi)]$, and the assertion follows.
\end{proof}
\ \\
This will enable us to reduce certain cases to previous ones:
E.g., if a maximal overgroup of $x$ has a (non-abelian simple) Lie-type socle over the same characteristic as the group $S$, with conjugates of $x$ contained only in the socle, then we can descend to extensions of maximal subgroups of this socle.
\\ \\
\begin{itemize}
\item $S = PSL_n(q)$, and $|M|$ divides $|GL_{n/r}(q^r).r|$.\footnote{see \cite{KL}, chapter 4, for the precise isomorphy types of the groups $M$} Here we lose a primitive prime divisor of $q^{n-1}-1$, unless $n \leq 3$, $(n,q) = (7,2)$, or the case that $r = n$ is this prime divisor. In the last case (with $n>3$) we certainly lose a primitive prime divisor of $q^{n-2} - 1$, and $PSL_7(2)$ can be checked directly.\\
If $n=3$, $|M|$ divides $3(q^3-1)$, and therefore cannot contain the homomorphic image of a cyclic torus of $SL_n(q)$ of order $q^2-1$, unless $\frac{q+1}{(3, q-1)} \ | \ 3$, i.e. $q=2$; but in $PSL_3(2)$ we have already found a different non-QSI character.\\
Finally, if $n=2$ the dihedral group $M$ contains no elements of order $\frac{q-1}{(2,q-1)}$, unless $q \leq 5$, in which case we end up with $S = A_5$ or a solvable group.
\item Similarly, the subgroups of type $O_{2n/r}^-(q^r)$ or $GU_{n}(q)$ in $S = P\Omega_{2n}^-(q)$ cannot contain a primitive prime divisor of $q^{2n-2}-1$, and for the only Zsigmondy exception $(n,q) = (4,2)$ we can check directly that no maximal subgroup has order divisible by 7 and 17.
\item For $S = P\Omega_{2n+1}(q)$, $M$ of type $O_1(q) \bot O_{2n}^-(q)$, there are homomorphic images of cyclic tori of order $q^n-1$ that are no longer contained in $M$.\\
(Alternatively, in $M = \Omega_{2n}^{-}(q).2$ we can use the descent argument mentioned above\footnote{note that the torus $\langle x \rangle$ contains the center of $\Omega_{2n}^{-}(q)$, so the overgroups of $x$ correspond to those considered above for $P\Omega_{2n}^{-}(q)$}, and reduce to the previous case)
\item In case $PSp_{2n}(q)$, the $\mathcal{C}_3$-subgroups lose primitive prime divisors of $q^{2n-2}-1$, with Zsigmondy exceptions only for $n=2$ or $(n,q) = (4,2)$. Here we only have to consider the groups of type $Sp_{2n/r}(q^r)$, as the other type appears only for $n$ odd. Now $M = PSp_4(4).4 \leq PSp_8(2)$ loses the prime factor 7, while $PSp_2(q^2).2 \cong PSL_2(q^2).2$ always has a non-abelian simple socle, hence the usual descent argument applies.\\
The $\mathcal{C}_8$-case yields $M = \Omega_{2n}^{-}(q).2$, as in the previous case.
\item In $PSU_n(q)$, with $n$ odd, the $C_3$-subgroups $M$ lose a primitive prime divisor of $q^{2(n-2)}-1$, with Zsigmondy exceptions only for $(n,q) = (5,2)$ (which once again can be excluded directly), or $n = 3$ and $q$ a Mersenne prime. But here $(|M|, q+1) = (3 \frac{q^3+1}{q+1}, q+1)$ divides 3, so for QSI-property, $q+1$ would have to be a power of 3, which cannot hold for $q$ a Mersenne prime.\\
If $n$ is even, the $\mathcal{C}_1$-subgroups $M$ lose a factor $q^n-1$, i.e. a primitive prime divisor of $q^n-1$ or $q^{n/2}-1$ is missing, depending on whether 4 does or does not divide $n$. There are no Zsigmondy exceptions here as the exponent is either odd or divisible by 4.
\item $S = P\Omega_n^+(q)$, the groups of type $O_1(q) \bot O_{2n-1}(q)$ or $O_2^-(q) \bot O_{2n-2}^-(q)$ can be ruled out by the descent argument, as conjugates of $x$ are only contained in the nonsolvable component. The same argument applies for the exceptional cases in $S = P\Omega_8^+(q)$.\\
In the $\mathcal{C}_3$-subgroups of type $GU_{n}(q)$ (which occur only for $n$ even), we lose a primitive prime divisor of $q^{n-1}-1$, and similarly in those of type $O_n(q^2)$ (only for $n$ odd) a primitive prime divisor of $q^{2(n-2)}-1$, with the one exception $(n,q) = (5,2)$, where we can check directly that no proper subgroup has order divisible by 31 and 17.
\end{itemize}
\vspace{10mm}
\section{Proof for exceptional Lie type groups}
For the exceptional Lie type groups we choose elements $x$ as generators of certain maximal tori. To be more precise, we follow Weigel (\cite{W1}, table 1), and choose the cyclic maximal tori given there together with their respective normalizers \footnote{for the cyclic structure of these tori see e.g. the lists in \cite{Ka2}, chapter 2}. Let $\widehat{x}$ be the corresponding element in the simply-connected version $\widehat{S} = S_{sc}$ of $S$ (which is in many cases the same as $S$ here). \\
\begin{table}
\caption{Choice of elements in exceptional Lie type groups}
\label{exctab}
\renewcommand{\arraystretch}{1.3}
		\begin{tabular}[h]{|l|c|}
		\hline
			{S} & {$ord(\widehat{x})$}\\
			\hline
			${^2}B_2(q)$ & $q + \sqrt{2q} + 1$ \\
			${^2}G_2(q)$ & $q + \sqrt{3q} + 1$ \\
			${^2}F_4(q)$ & $q^2 + \sqrt{2q^3} + q + \sqrt{2q} + 1$ \\
			$G_2(q)$ & $q^2-q +1$ \\
			${^3}D_4(q)$ & $q^4-q^2+1$ \\
			$F_4(q)$ &  $q^4-q^2+1$ \\
			$E_6(q)$ & $q^6+q^3+1$ \\
			${^2}E_6(q)$ & $q^6-q^3+1$ \\ 
			$E_7(q)$ & $(q+1)(q^6-q^3+1)$ \\
			$E_8(q)$ & $q^8+q^7-q^5-q^4-q^3+q+1$ \\
			\hline
		\end{tabular}
		\end{table}	
\vspace{4mm}	
\begin{lemma} Let $S$ be one of the groups $^3D_4(q)$, $^2B_2(q)$ ($q=2^{2n+1}, n \geq 1$), $^2F_4(q)$ ($q=2^{2n+1}, n \geq 1$), $^2G_2(q)$ ($q=3^{2n+1}, n \geq 1$), $G_2(q)$ ($q>2$).\\
Let $S^m \leq G \leq Aut(S^m)$.\\
Then $G$ is QSI.
\end{lemma}
\begin{proof}
For all these groups the structure of maximal subgroups is known (see \cite{Co1} (Theorems 2.3 and 2.4), \cite{Kl1} (Theorems A and C), \cite{Kl2} (Theorem 1), \cite{Ma1} (Main Theorem), \cite{Su1} (Theorem 9)).\\
\\
In the cases ${^3}D_4$, ${^2}B_2$, ${^2}G_2$ and ${^2}F_4$, due to the choice of $x$ and Zsigmondy, the only maximal overgroup of $x$ is $N_S(x)$.\\
This group has the following order:
\begin{itemize}
	\item $4 \cdot ord(x)$ in ${^3}D_4(q)$ and ${^2}B_2(q)$
	\item $6 \cdot ord(x)$ in ${^2}G_2(q)$
	\item $12 \cdot ord(x)$ in ${^2}F_4(q)$.\\ 
\end{itemize}
Once again it's easy to see that certain primes other than the field characteristic cannot be contained (any divisor of $q-1$ in case ${^2}B_2(q))$, any odd divisor of $q^6-1$ in case ${^3}D_4(q)$, any odd divisor of $q-1$ in case ${^2}G_2(q)$, any divisor of $q^3+1$ in case ${^2}F_4(q)$).
In case $G_2(q)$, for $n \geq 5$, \cite{W1}, 4d), yields $SU_3(q).2$ as the only maximal subgroup of $x$, and here no odd prime divisor of $q^3-1$ is contained.
\\
This leaves the cases $q \in \left\{3,4\right\}$.\\
One checks directly that $G_2(4)$ has no maximal subgroups containing each of the prime divisors 3, 5, 7 and 13.
$G_2(3)$ does have maximal subgroups containing all the necessary prime divisors, namely such of type $PSL_2(13)$, but elements of order 8 are missing here.
\end{proof}
\vspace{5mm}
\begin{lemma}
Let $S$ be one of $E_6(q)$, $^2E_6(q)$, $ E_7(q)$, $E_8(q)$, $F_4(q)$.\\
Let $S^m \leq G \leq Aut(S^m)$.\\
Then $G$ is not QSI.
\end{lemma}
\begin{proof}
We gather from \cite{W1}, chapter 4, f)-j), (or, for the few cases excluded in \cite{W1}: from \cite{GK}, table 4) the maximal subgroups  $M$ of $S$ containing a cyclic torus $\left\langle x\right\rangle$, as chosen in table \ref{exctab}. \\
In case $S = F_4(q)$, for $q \neq 2$, it always holds that $M \cong {^3}D_4(q).3$.\\
In case $S = E_6(q)$ we have $|M| = \frac{1}{(3,q-1)} \cdot |SL_3(q^3).3| = 3 \cdot q^{9} \cdot (q^9-1)(q^6-1)$ \ 
In case $S = {^2}E_6(q)$ we get $|M| = \frac{1}{(3,q+1)} \cdot |SU_3(q^3).3| = 3 \cdot q^{9} \cdot (q^9+1)(q^6-1)$.
\\
For $S = E_7(q)$, and $q \neq 2$, we get $|M| = \frac{1}{(2,q-1)} \cdot |(\mathbb{Z}_{q+1}.{^2E}_6(q)).2|$, \\ 
and in case $S = E_8(q)$ it always holds that $|M| = |N_G(\left\langle x\right\rangle)| = 30 \cdot |\left\langle x\right\rangle|$.
\\
In each case one easily obtains some missing prime divisors just by checking the orders of the groups $S$.
\\
\\
The last remaining cases are $S = F_4(2)$ and $E_7(2)$.\\
Here we deviate from the standard choice in table \ref{exctab} for our choice of elements $x$:\\
In $E_7(2)$, by \cite{GK}, tab.4, there is a conjugacy class of elements $x$ of order $129 = 2^7+1$, contained only in maximal subgroups of type $SU_8(2) \cong PSU_8(2)$ - namely Singer elements of the unitary group $GU_7(2) \leq SU_8(2)$.\\
Analogously, in $F_4(2)$ there is a conjugacy class of self-centralizing elements of order $17 = 2^4+1$, contained only in maximal subgroups of type $Sp_8(2) \cong PSp_8(2)$  (namely Singer elements of these groups).\\
\\
As the subgroups are also over characteristic 2, we can use the results from the "'classical group"' section.
\end{proof}
\vspace{6mm}
Now, by the classification of finite simple groups, all cases have been treated, and theorem \ref{Hauptsatz} has been proven.\\
\\
After checking the few cases where a simple group can be a Lie type group in different defining characteristics, or an alternating and Lie type group at the same time (cf. \cite{MSW}, Prop. 2.9.1), we can also summarize the results about Steinberg characters in the following
\begin{lemma}
Let $S = S(n, p^r)$ be a finite non-abelian simple group of Lie type in defining characteristic $p$. Then its Steinberg character is QSI if and only if $S = PSL_2(5) (\cong A_5)$, $PSL_2(7)$, $PSL_3(2)$ or $PSp_4(3)$.\footnote{This last case is in fact an example for a properly quasi-monomial character, i.e. $2St$ is monomial, but $St$ is not}.
\end{lemma}
\vspace{15mm}
\newpage
\renewcommand{\appendixname}{Appendix: Finite groups of Lie type}
\begin{appendix}
\section{Finite groups of Lie type}
In table \ref{T2} we summarize the orders of the simply-connected versions $L_{sc}$ of the finite groups of Lie type, and the orders of their centers respectively. The group $L_{sc}/Z(L_{sc})$ is then, with only a few excpeptions, a simple group.
\\
\begin{table}
\caption{Finite groups of Lie type}
\label{T2}
\renewcommand{\arraystretch}{1.3}
		\begin{tabular}[h]{|l|c|c|}
		\hline
			{$L_{sc}/Z(L_{sc})$} & {$|L_{sc}|$} & {$|Z(L_{sc})|$}\\
			\hline
			$PSL_n(q)$ & $q^{\frac{n(n-1)}{2}} \cdot \prod_{i=2}^n{(q^i-1)}$ & $(n,q-1)$\\
			$PSp_{2n}(q)$ & $q^{n^2} \cdot \prod_{i=1}^n{(q^{2i}-1)}$ & $(2,q-1)$\\
			$PSU_n(q)$ & $q^{\frac{n(n-1)}{2}} \cdot \prod_{i=2}^n{(q^i-(-1)^i)}$ & $(n,q+1)$\\
			$P\Omega_{2n+1}(q)$ & $q^{n^2} \cdot \prod_{i=1}^n{(q^{2i}-1)}$ & $(2,q-1)$\\
			$P\Omega_{2n}^-(q)$ & $q^{n(n-1)} \cdot (q^n+1) \cdot \prod_{i=1}^{n-1}{(q^{2i}-1)}$ & $(4,q^n+1)$\\
			$P\Omega_{2n}^+(q)$ & $q^{n(n-1)} \cdot (q^n-1) \cdot \prod_{i=1}^{n-1}{(q^{2i}-1)}$ & $(4,q^n-1)$\\
			\hline
			${^2}B_2(q)$ & $q^2(q^2+1)(q-1)$, \ $q=2^{2n+1}$ & $1$\\
			${^2}G_2(q)$ & $q^3(q^3+1)(q-1)$, \ $q=3^{2n+1}$ & $1$\\
			${^2}F_4(q)$ & $q^{12}(q^6+1)(q^4-1)(q^3+1)(q-1)$, \ $q=2^{2n+1}$ & $1$\\
			$G_2(q)$ & $q^6(q^6-1)(q^2-1)$ & $1$\\
			${^3}D_4(q)$ & $q^{12}(q^8+q^4+1)(q^6-1)(q^2-1)$ & $1$\\
			$F_4(q)$ &  $q^{24}(q^{12}-1)(q^8-1)(q^6-1)(q^2-1)$ & $1$\\
			$E_6(q)$ & $q^{36}(q^{12}-1)(q^9-1)(q^8-1)(q^6-1)(q^5-1)(q^2-1)$ & $(3,q-1)$\\
			${^2}E_6(q)$ & $q^{36}(q^{12}-1)(q^9+1)(q^8-1)(q^6-1)(q^5+1)(q^2-1)$ & $(3,q+1)$\\ 
			$E_7(q)$ & $q^{63}(q^{18}-1)(q^{14}-1)(q^{12}-1)(q^{10}-1)(q^8-1)(q^6-1)(q^2-1)$ & $(2,q-1)$\\
			$E_8(q)$ & $q^{120}(q^{30}-1)(q^{24}-1)(q^{20}-1)(q^{18}-1)(q^{14}-1)(q^{12}-1)(q^8-1)(q^2-1)$ & $1$\\
			\hline
		\end{tabular}
		\end{table}
		\\
\vspace{4mm}
\end{appendix}


\begin{thebibliography}{9}
\bibitem{As1}
\textbf{M. Aschbacher}\
\textit{On the maximal subgroups of the finite classical groups}\\
Invent. Math. 76 (1984), 469-514.
\bibitem{Be1}
\textbf{A. Bereczky}\
\textit{Maximal Overgroups of Singer Elements in Classical Groups}\\
Journal Of Algebra 234 (2000), 187-206.
\bibitem{Ca2}
\textbf{R.W. Carter}\
\textit{Centralizers of Semisimple Elements in the Finite Classical Groups}\\
Proc. London Math. Soc. 42 (1981), 1-41.
\bibitem{ATLAS}
\textbf{J.H. Conway, R.T. Curtis, S.P. Norton, R.A. Parker, R.A. Wilson}\
\textit{An Atlas of Finite Groups}\\
Oxford University Press, 1985
\bibitem{Co1}
\textbf{B. Cooperstein}\
\textit{Maximal subgroups of $G_2(2^n)$}\\
Journal of Algebra 70 (1981), 23-36.
\bibitem{Fe}
\textbf{W. Feit}\
\textit{Steinberg Characters}\\
Bull. London Math. Soc. 1995 27, 34-38.
\bibitem{GK}
\textbf{R.M. Guralnick, W.M. Kantor}\
\textit{Probabilistic Generation of Finite Simple Groups}\\
Journal Of Algebra 234 (2000), 743-792.
\bibitem{H3}
\textbf{B. Huppert}\
\textit{Singer-Zyklen in klassischen Gruppen}\\
Math. Zeitschrift 117 (1970), 141-150.
\bibitem{Is}
\textbf{I.M. Isaacs}\
\textit{Character Theory of Finite Groups}\\
AMS Chelsea Publishing (2006)
\bibitem{Ka2}
\textbf{W.M. Kantor, A. Seress}\
\textit{Prime power graphs for groups of Lie type}\\
J. Alg. 247 (2002), 370-434.
\bibitem{Kl1}
\textbf{P. Kleidman}\
\textit{The maximal subgroups of the Chevalley groups $G_2(q)$ with q odd, the Ree groups ${^2}G_2(q)$, and their automorphism groups}\\
J. Alg. 117 (1988), 30-71.
\bibitem{Kl2}
\textbf{P.Kleidman}\
\textit{The maximal subgroups of the finite Steinberg triality groups ${^3}D_4(q)$ and their automorphism groups}\\
J. Alg. 115 (1988), 182-199.
\bibitem{KL}
\textbf{P. Kleidman, M. Liebeck}\
\textit{The Subgroup Structure of the Finite Classical Groups}\\
London Math. Soc. Lecture Note Series, Vol. 129, Cambridge Univ. Press, Cambridge, 1990. 
\bibitem{La}
\textbf{R. Langlands}\
\textit{Base Change for GL(2)}\\
Annals of Math. Studies, Vol. 96, 1980.
\bibitem{Ma1}
\textbf{G. Malle}\
\textit{The maximal subgroups of \ ${^2}F_4(q^2)$}\\
J. Alg. 139 (1991), 52-69.
\bibitem{MSW}
\textbf{G. Malle, J. Saxl, T. Weigel}\
\textit{Generation of classical groups}\\
Geom. Dedicata 49 (1994), 85-116.
\bibitem{Flo}
\textbf{F. M\"oller}\
\textit{\"Uber die QSI-Eigenschaft einiger einfacher Gruppen}\\
Diploma thesis, Univ. W\"urzburg, 2005
\bibitem{Mu1}
\textbf{M. R. Murty}\
\textit{An Introduction to Artin L-functions}\\
J. of the Ramanujan Mathematical Society, Vol. 16, 2001.
\bibitem{Su1}
\textbf{M. Suzuki}\
\textit{On a class of doubly transitive groups}\\
Ann. Math. 75 (1962), 105-145.
\bibitem{Tu}
\textbf{J. Tunnell}\
\textit{Artin's conjecture for representations of octahedral type}\\
Bulletin of the Amer. Math. Society, Vol. 5, 1981, 173-175.
\bibitem{W1}
\textbf{T. Weigel}\
\textit{Generation of exceptional groups of Lie-type}\\
Geom. Dedicata 41 (1992), 63-87.
\bibitem{Z1}
\textbf{K. Zsigmondy}\
\textit{Zur Theorie der Potenzreste}\\
Journal Monatshefte f\"ur Mathematik 3 (1892), 265-284.
\end{thebibliography}
\end{document}